\documentclass[11pt]{article}

\usepackage{amsmath, amssymb, amsthm} 
\usepackage{centernot}
\usepackage{mathpazo}
\usepackage[margin=2.5cm]{geometry}
\usepackage{color, graphicx, soul}
\usepackage{subfig}
\usepackage{adjustbox}

\usepackage{resizegather}

\usepackage{xargs}
\usepackage[colorinlistoftodos,prependcaption,textsize=footnotesize]{todonotes}
\newcommandx{\attn}[2][1=]{\todo[linecolor=red,backgroundcolor=blue!25,bordercolor=red,#1]{#2}}
\newcommandx{\other}[2][1=]{\todo[linecolor=OliveGreen,backgroundcolor=OliveGreen!25,bordercolor=OliveGreen,#1]{#2}}
\newcommandx{\thiswillnotshow}[2][1=]{\todo[disable,#1]{#2}}

\usepackage{hyperref}

\usepackage{empheq}
\definecolor{myblue}{rgb}{.9, .9, 1}

\usepackage[capitalize]{cleveref}
\crefname{equation}{}{equations}
\crefname{chapter}{Appendix}{chapters}
\crefname{item}{}{items}
\crefname{figure}{Figure}{figures}
\makeatletter
\def\namedlabel#1#2{\begingroup
   \def\@currentlabel{#2}%
   \label{#1}\endgroup
}
\makeatother

\makeatletter
\def\th@plain{%
  \thm@notefont{}
  \itshape 
}
\def\th@definition{%
  \thm@notefont{}
  \normalfont 
}
\makeatother

\newtheorem{theorem}{Theorem}[section]
\newtheorem{lemma}[theorem]{Lemma}
\newtheorem{corollary}[theorem]{Corollary}
\newtheorem{proposition}[theorem]{Proposition}
\newtheorem{fact}[theorem]{Fact}

\theoremstyle{definition}

\theoremstyle{definition}
\newtheorem{example}[theorem]{Example}
\theoremstyle{definition}
\newtheorem{remark}[theorem]{Remark}

\usepackage[shortlabels]{enumitem}
\setlist[enumerate]{nosep}

\allowdisplaybreaks

\newcommand{\scal}[2]{\left\langle {#1},{#2} \right\rangle}

\newcommand{\To}{\ensuremath{\rightrightarrows}}

\newcommand{\NN}{\ensuremath{\mathbb N}}

\newcommand{\RR}{\ensuremath{\mathbb R}}

\newcommand{\inte}{\ensuremath{\operatorname{int}}}

\newcommand{\conv}{\ensuremath{\operatorname{conv}}}

\newcommand{\ran}{\ensuremath{\operatorname{ran}}}

\newcommand{\dom}{\ensuremath{\operatorname{dom}}}

\newcommand{\epi}{\ensuremath{\operatorname{epi}}}

\newcommand{\Id}{\ensuremath{\operatorname{Id}}}
\newcommand{\ent}{\ensuremath{\operatorname{ent}}}

\newcommand{\lev}[1]{\ensuremath{\mathrm{lev}_{\leq #1}\:}}
\newcommand{\fenv}[2][]%
{\ensuremath{\,\overrightarrow{\operatorname{env}}_{#2}}^{#1}}
\newcommand{\benv}[2][]%
{\ensuremath{\,\overleftarrow{\operatorname{env}}_{#2}^{#1}}}
\newcommand{\env}[2][]%
{\ensuremath{{\operatorname{env}}_{#2}^{#1}}}
\newcommand{\W}{\mathcal W}
\newcommand{\bP}[2][]%
{\ensuremath{\,\overleftarrow{\operatorname{P}}_{#2}^{#1}}}

\newcommand{\fP}[2][]%
{\ensuremath{\,\overrightarrow{\operatorname{P}}_{#2}^{#1}}}

\usepackage{xcolor}
\usepackage{tikz}
\usetikzlibrary{calc}
\usetikzlibrary{decorations.pathmorphing}

\begin{document}

\title{The Generalized Bregman Distance}

\author{
Regina S.\ Burachik\thanks{
Mathematics, UniSA STEM, University of South Australia, Mawson Lakes, SA 5095, Australia. 
E-mail: \texttt{regina.burachik@unisa.edu.au}.},~
Minh N.\ Dao\thanks{
School of Engineering, Information Technology and Physical Sciences, Federation University Australia, Ballarat, VIC 3353, Australia.
E-mail: \texttt{m.dao@federation.edu.au}.},~~~and~
Scott B.\ Lindstrom\thanks{
Department of Applied Mathematics, Hong Kong Polytechnic University, Hong Kong.
E-mail: \texttt{scott.b.lindstrom@polyu.edu.hk}.}
}

\date{September 26, 2020}

\maketitle

\begin{abstract} \noindent
Recently, a new distance has been introduced for the graphs of two point-to-set operators, one of which is maximally monotone. When both operators are the subdifferential of a proper lower semicontinuous convex function, this distance specializes under modest assumptions to the classical Bregman distance. We name this new distance the \emph{generalized Bregman distance}, and we shed light on it with examples that utilize the other two most natural representative functions: the Fitzpatrick function and its conjugate. We provide sufficient conditions for convexity, coercivity, and supercoercivity: properties that are essential for implementation in proximal point type algorithms. We establish these results for both the left and right variants of this new distance. We construct examples closely related to the Kullback--Leibler divergence, which was previously considered in the context of Bregman distances, and whose importance in information theory is well known. In so doing, we demonstrate how to compute a difficult Fitzpatrick conjugate function, and we discover natural occurrences of the Lambert $\W$ function, whose importance in optimization is of growing interest.  
\end{abstract}

{\small 


}

\section{Introduction}

Throughout, unless stated otherwise, $\left( X,\left\Vert \cdot \right\Vert \right) $ is a real Banach space with dual $\left( X^\ast,\left\Vert \cdot \right\Vert_\ast \right)$, and $\Gamma_0(X)$ is the set of all proper lower semicontinuous convex functions from $X$ to $\mathbb{R}_{\infty} :=\left]-\infty, +\infty\right]$. 

In 1967, Bregman introduced the distance constructed for a differentiable convex function $f$, 
\begin{equation} 
\label{e:Df}
\mathcal{D}_f \colon X \times X \to \left[0, +\infty\right] \colon (x,y) \mapsto
\begin{cases}
f(x) -f(y) -\scal{\nabla f(y)}{x -y} &\text{~if~} y\in \inte \dom f, \\
+\infty &\text{~otherwise},
\end{cases}
\end{equation} 
which now bears his name \cite{Bre67} and whose corresponding envelopes and proximity operators specify to the Moreau proximity operator \cite{Attouch77,Attouch84,Mor65} and envelope when $f$ is the energy $\left \| \cdot \right \|^2/2$. The study of Bregman distances has become popular, following their 1981 reintroduction by Censor and Lent \cite{Censor1981}. Bregman functions were introduced to solve feasibility problems, and then used for producing more general versions of the classical proximal-point algorithm, both for convex minimization and for monotone variational inequality problems. The use of Bregman distances in prox-like methods for convex minimization is found in \cite{BC19,BCN06,ByrCen01,CH02,Censor1992,Censor1997,CT93,Kiw97,Reich2010,Reem2019}, while the use for monotone variational inequalities can be found in \cite{BurDut10,BI98,BI99,BurKas12,BS01,Eck93}. Properties of Bregman functions have been the focus of research since the late 1990s, see \cite{BB97,BBC01,CIZ98}. A good, brief bibliographic overview of their history is found in \cite[p.1233]{Facchinei}, and the book of Censor and Zenios is also instructive \cite{Censor1997}. When $f$ is not the energy, the distance may fail to be symmetric, and so one is led to consider the left and right versions of envelopes and their associated proximity operators. The asymptotic properties of Bregman envelopes with respect to a parameter were explored in \cite{BDL17}.  

From \cite[Proposition~3.2]{BB97}, we have a dual characterization of the Bregman distance for a differentiable convex function $f$:
\begin{subequations}
\begin{align}
\mathcal{D}_f(x,y) &= f(x)-f(y)-\langle \nabla f(y),x-y \rangle \label{FY1}\\
&=f(x)+f^*(\nabla f(y))- \langle \nabla f(y),x \rangle \label{FY2}
\end{align}
\end{subequations} 
Here \eqref{FY1} is the definition of the Bregman distance and \eqref{FY2} uses the Fenchel--Young equality: $(\forall v \in \nabla f(y))\; f(y)+f^*(v)=\langle y, v\rangle$. From \eqref{FY2}, they made the observation that 
\begin{equation}
\mathcal{D}_f(x,y)=0 \iff \nabla f(x)=\nabla f(y),
\end{equation}
and so $\mathcal{D}_f$ has the dual characterization of serving as a distance between gradients\footnote{To show $\impliedby$, simply substitute $\nabla f(x)$ for $\nabla f(y)$ and apply Fenchel--Young once more, this time for the variable $x$ instead of $y$.}. Based on this characterization, they introduced a distance based on the representative function $h$ of a monotone operator:
\begin{subequations}
\begin{align}
\mathcal{D}_T^{\flat,h}(x,y) &:=\inf_{v\in Ty}\left( h(x,v)-\langle
x,v\rangle \right),\\ 
\mathcal{D}_T^{\sharp,h}(x,y) &:=\sup_{v\in Ty}\left( h(x,v)-\langle
x,v\rangle \right).
\end{align}
\end{subequations} 
This distance generalizes the Bregman distance, specializing---under the mild domain conditions in \ref{ss:generalizesbregman}---thereto when $h$ is the \emph{Fenchel--Young representative} for $T=\nabla f$, which is defined by $f \oplus f^\ast(x,y)=f(x)+f^\ast(y)$.

Naturally, we name this more general distance the \emph{generalized Bregman distance (GBD)}. In lieu of the Fenchel--Young representative, the Fitzpatrick function and its conjugate are the two other functions that are most natural to consider. As with the Bregman distance, we obtain left and right variants; these admit new left and right coercivity and supercoercivity properties, along with envelopes and corresponding proximity operators, when the GBD replaces the Bregman distance in the construction of envelopes.

\subsection{Outline and contributions}

This article is outlined as follows. In Section~\ref{sec:preliminaries}, we recall the GBD as introduced in \cite{BM-L18}. We provide its basic properties and clarify the domain conditions under which the Fenchel--Young representative case specializes to the Bregman distance. We also introduce the closed variant, which may specialize to the Bregman distance at more points on the boundary of the domain.

In Section~\ref{sec:examples}, we show how to compute the new GBDs. We first illustrate with the energy, which specifies to the Moreau case when the Fenchel--Young representative is used. We also illustrate with the Boltzmann--Shannon entropy, whose derivative is the logarithm and whose Bregman distance, the Kullback--Leibler divergence, is commonly used as a measure of the difference between positive vectors in information theory and elsewhere. We compare the Kullback--Leibler divergence with the similar Fenchel--Young representative GBD for the logarithm, and also with its closed version. We also illustrate how to compute with the two other most natural representatives to consider: the book-end cases for the representative function set. These are the smallest representative, named the \emph{Fitzpatrick function}, and the biggest representative, which is obtained using the conjugate of the Fitzpatrick function.

Interestingly, while the Fitzpatrick function for the logarithm was discovered in \cite{BMcS06}, the present work contains the first computation of its conjugate. The discovery and proof rely upon the graphical characterizations of representative functions, and the special function Lambert $\W$ plays an important role in the computational aspects of discovery. The way that we tackle this problem is very prototypical of the approach that one might need to use when computing other representative functions and GBDs. We furnish a full discussion of the process in Section~\ref{subsec:BSentropy}, so that it may serve as a tutorial for other researchers.

Section~\ref{sec:properties} contains our most important results. We provide a framework of sufficient conditions for coercivity and supercoercivity of the left and right GBDs. This framework uses the fact that the GBDs majorize a set distance. We illustrate with figures in 2 dimensions, and we provide examples of what may go wrong when the sufficient conditions provided by our framework are not satisfied. In Section~\ref{sec:coercivityofthesum}, we explain how these coercivity and supercoercivity conditions may be used to guarantee the coercivity of the sum of the distance together with a Legendre function. 

Such sums are the basis of the corresponding envelope functions and their proximity operators. In the Bregman case, these coercivity conditions admit the further analysis of the envelopes and proximity operators, including their asymptotic behaviour as the scalar parameter varies \cite{BDL17}. Our work lays the necessary foundation for such an analysis in the case of envelopes built from GBDs. The study of envelopes is important, because many optimization algorithms may be viewed as special cases of gradient descent applied to envelopes; see, for example, \cite{Patrinos2014,Themelis2018}. We conclude in Section~\ref{sec:conclusion}.

\section{Preliminaries on generalized Bregman distances}\label{sec:preliminaries}

Given a function $f\colon X\to \mathbb{R}_{\infty}$, its \emph{domain} (or \emph{effective domain}) is defined by $\dom f :=\{x\in X: f(x) <\infty\}$ and its \emph{lower level set} at height $\xi\in \mathbb{R}$ by $\lev{\xi} f :=\{x\in X: f(x)\leq \xi\}$. 
The function $f$ is said to be \emph{proper} if $\dom f\neq \varnothing$; \emph{lower semicontinuous (lsc)} at $\bar{x}$ if $f(\bar{x})\leq \liminf_{x\to \bar{x}} f(x)$; \emph{convex} if  
\begin{equation}
\forall x,y\in X,\ \forall \lambda\in [0,1],\quad 
f((1-\lambda)x+\lambda y)\leq (1-\lambda)f(x) +\lambda f(y);
\end{equation}
\emph{coercive} if $\lim_{\|x\|\to \infty} f(x) =\infty$; and \emph{supercoercive} if $\lim_{\|x\|\to \infty} f(x)/\|x\| =\infty$. 

Let $f\colon X\to \mathbb{R}_{\infty}$ be proper. The \emph{subdifferential} of $f$ is the point-to-set mapping $\partial
f:X\rightrightarrows X^{\ast }$ defined by 
\begin{equation}\label{e:partialf}
\partial f(x):=
\begin{cases}
\{v\in X^\ast: \forall y\in X,\; \langle y-x,v \rangle +f(x)\leq f(y)\} & \text{if~} x\in \dom f,\\ 
\varnothing & \text{otherwise}.%
\end{cases}
\end{equation}
The \emph{Fenchel conjugate} of $f$ is the mapping
\begin{equation}
f^\ast\colon X^\ast\to \mathbb{R}_{\infty}\colon v\mapsto \sup_{x\in X}\{\langle x,v\rangle -f(x)\}.
\end{equation} 
From the definition, we have the \emph{Fenchel-Young inequality}
\begin{equation}
\forall (x,v)\in X\times X^\ast,\quad f(x) +f^\ast(v)\geq \langle x,v \rangle,
\end{equation}
and if $f$ is convex, then 
\begin{equation}
f(x) +f^\ast(v) =\langle x,v \rangle \iff v\in \partial f(x).
\end{equation}

Given a point-to-set operator $T\colon X\rightrightarrows X^{\ast}$, its \emph{domain} is $\dom T :=\{x\in X: Tx\neq \varnothing\}$, its \emph{range} is $\ran T :=T(X)$, and its \emph{graph} is $\mathcal{G} (T) :=\{(x,x^{\ast})\in X\times
X^{\ast}: x^{\ast}\in Tx\}$. Additionally, $T$ is said to be maximally monotone if 
\begin{equation}\label{d:maximallymonotone}
(x,u) \in \mathcal{G} (T) \quad \iff \quad (\forall (y,v)\in \mathcal{G} (T)) \quad \langle x-y,u-v \rangle \geq 0.
\end{equation}
A detailed study of maximally monotone operators can be found in \cite[Chapters~20 and 21]{BC17} for Hilbert spaces, and in \cite[Chapter~4]{BI08} for the Banach space case.

\subsection{Representative functions}

Let $S\colon X\rightrightarrows X^\ast$ be a maximally monotone operator. We recall from \cite[Definition~2.3]{BM-L18} that $h\colon X\times X^\ast\to \mathbb{R}_{\infty}$ \emph{represents $S$} and denote $h\in \mathcal{H}(S)$ if the following three conditions hold:
\begin{enumerate}
\item[(a)]
$h$ is convex and norm $\times $ weak$^{\ast }$ lower semicontinuous in $X\times X^\ast$.
\item[(b)]
$\forall (x,v)\in X\times X^\ast,\ h(x,v)\geq \langle x,v\rangle$.
\item[(c)]
$h(x,v)=\langle x,v\rangle \Longleftrightarrow (x,v)\in \mathcal{G}(S)$.
\end{enumerate}

We will make use, in particular, of several representative functions. These are as follows.
\begin{enumerate}
\item The Fitzpatrick function $F_{S}:(x,y) \mapsto \sup_{(z,w) \in \mathcal{G}(S)} \left( \langle z-x,y-w \rangle +\langle x,y\rangle \right)$ is the smallest member of $\mathcal{H}(S)$; see Fitzpatrick's 1998 paper \cite[Theorem~3.7]{Fitzpatrick88}.
\item The largest member of $\mathcal{H}(S)$ we denote by $\sigma_{S}$; it may be computed using the identities in Fact~\ref{fact:handy}\ref{handy:sigmafitz}.
\item The Fenchel--Young representative $f\oplus f^\ast \in \mathcal{H}(\partial f)$, where $f\in \Gamma_0(X)$ and $f\oplus f^\ast\colon X\times X^\ast\to \mathbb{R}_{\infty}$ is defined by
\begin{equation}
\forall (x,v)\in X\times X^\ast,\quad f\oplus f^\ast(x,v) :=f(x) +f^\ast(v).
\end{equation}
\end{enumerate}

The Fitzpatrick function has proven quite useful in monotone operator theory; see, for example, \cite{BC17,Borwein2011b}.

\begin{fact}\label{fact:handy}  
Let $S:X\To X^*$ be a maximally monotone operator and $X$ a real Banach space. We have the following characterizations of $\sigma_S$ and $F_S^*$.
\begin{enumerate}
\item\label{handy:sigma} From \cite[Equation~(33)]{bs2002}, we have that
\begin{equation}
\epi \sigma_S = \overline{\rm co}\left(\epi \left(\pi + \iota_{G(S)}\right)\right), \quad\text{where~} \pi: (p,p^*) \mapsto \langle p,p^* \rangle
\end{equation}
is the duality product defined in $X\times X^*$, and $\overline{\rm co}(A)$ is an abbreviation for the closure of the convex hull of a set $A$.
\item\label{handy:fitz} From \cite[Proposition~10.56]{BC17}\footnote{Note that the setting in the exposition \cite{BC17} is a Hilbert space, although the veracity of \ref{handy:fitz} in a Banach space follows from \ref{handy:sigma} and \ref{handy:sigmafitz}. The setting of \cite{bs2002} is a Banach space.}, we have that $F_S = \left(\iota_{G(S^{-1})}+ \pi \right)^*$.
\item\label{handy:sigmafitz} From \cite[Equation~(39)]{bs2002}, we have that $\sigma_{S}(x,y) = F_{S}^*(y,x)$.
\item\label{handy:sums} From \cite[Corollary~4.2]{bs2002}, if $h$ is convex and lower semicontinuous on $X \times X^*$ and $F_S \leq h \leq \sigma_S$, then $h \in \mathcal{H}(S)$.
\end{enumerate}
The astute reader will notice that \ref{handy:sigmafitz} may be obtained by combining \ref{handy:sigma} and \ref{handy:fitz}, since
\begin{align*}
\left(\iota_{\mathcal{G}(S)}+ \pi \right)(y,x) &=\begin{cases}
\infty & \text{if~} (y,x) \notin \mathcal{G}(S),\\
\langle y,x\rangle & \text{otherwise}
\end{cases}\\
&=\begin{cases}
\infty & \text{if~} (x,y) \notin \mathcal{G}(S^{-1}),\\
\langle x,y \rangle & \text{otherwise}
\end{cases}\\ 
&= \left(\iota_{\mathcal{G}(S^{-1})}+ \pi \right)(x,y).
\end{align*}
Additionally, \ref{handy:sums} is quite pleasing, because it admits as representative functions the convex combinations of other representative functions.
\end{fact}

\subsection{A new ``generalized Bregman'' distance between point-to-set operators}\label{ss:generalizesbregman}

From now on, we assume that $S\colon X\rightrightarrows X^\ast$ is a maximally monotone operator, $h\in \mathcal{H}(S)$, and $T\colon X\rightrightarrows X^\ast$. Following \cite[Definition~3.1]{BM-L18}, for fixed $(x,y)\in \dom S\times \dom T$, we define 
\begin{subequations}\label{D}
\begin{align}
\mathcal{D}_T^{\flat,h}(x,y) &:=\inf_{v\in Ty}\left( h(x,v)-\langle
x,v\rangle \right), \label{Dsharp}\\ 
\mathcal{D}_T^{\sharp,h}(x,y) &:=\sup_{v\in Ty}\left( h(x,v)-\langle
x,v\rangle \right).\label{Dflat}
\end{align}
\end{subequations}
If $y\not\in \dom T$, then $\mathcal{D}_T^{\flat ,h}(x,y) =\mathcal{D}_T^{\sharp,h}(x,y):=+\infty $ for every $x\in X$. If $x\not\in \dom S$, then $\mathcal{D}_T^{\flat,h}(x,y) =\mathcal{D}_T^{\sharp,h}(x,y) :=+\infty$ for every $y\in X$. 
When $T$ is point to point, we simply write $\mathcal{D}_T^h :=\mathcal{D}_T^{\flat,h} =\mathcal{D}_T^{\sharp,h}.$ 

For our examples, $T= S = \partial f$ is point-to-point on $\inte \dom f$ in which case we simply write $\mathcal{D}_{h}$. Additionally, when employing a specific representative function, we will use the name of the representative function used in place of $h$. If a distance is of the form \eqref{Dflat} or \eqref{Dsharp} we call it a \emph{generalized Bregman distance} or \emph{GBD} for short. The GBD specializes to the Bregman distance under certain circumstances, which we now recall. To a proper and convex function $f\colon X\to \mathbb{R}_{\infty}$, we associate two Bregman distances (see \cite{Kiw97}) defined by
\begin{subequations}\label{d:Bregman}
\begin{align}
\mathcal{D}_f^\flat(x,y) :=& f(x)-f(y)+\inf_{v \in \partial f(y)} \langle y-x,v \rangle \label{d:Bregman_flat}\\
\text{and}\quad \mathcal{D}_f^\sharp(x,y) :=& f(x)-f(y)+\sup_{v \in \partial f(y)} \langle y-x,v \rangle.\label{d:Bregman_sharp}
\end{align}
\end{subequations}  

Burachik and Mart\'inez-Legaz observed that the GBD specializes to the Bregman distance in the case where the Fenchel--Young representative distance is used. The following proposition fills a minor omission from \cite[Proposition~3.5]{BM-L18}, namely that the Fitzpatrick distance specializes to the Bregman distance under the mild condition that $(x,y) \notin (\dom f \setminus \dom \partial f) \times \dom \partial f$. In the case when $\dom f \setminus \dom \partial f = \varnothing$, they are everywhere equal. We will see later in an example that when $f$ is the Boltzmann--Shannon entropy \eqref{def:entropy}, we have $\dom f \setminus \dom \partial f = \{0\}$, and the two distances fail to be equal on the set $\{(0,y) | \; y>0 \}$.

\begin{proposition}[The GBD generalizes the Bregman distance]\label{prop:fitz_bregman}
Let $f\in \Gamma_0(X)$. Then
\begin{equation}
\mathcal{D}_{f \oplus f^\ast}^{\flat}(x,y) = \mathcal{D}_f^\flat(x,y)
\quad\text{and}\quad
\mathcal{D}_{f \oplus f^\ast}^{\sharp}(x,y) = \mathcal{D}_f^\sharp(x,y)
\end{equation}
whenever $(x,y)\notin (\dom f\setminus \dom\partial f)\times \dom\partial f$.
\end{proposition}
\begin{proof}
If $y\notin \dom\partial f$, then $\mathcal{D}_{f \oplus f^\ast}^{\flat}(x,y) =\mathcal{D}_f^\flat(x,y) =\infty$. If $y\in \dom\partial f$ and $x\notin \dom f$, we also have that $\mathcal{D}_{f \oplus f^\ast}^{\flat}(x,y) =\mathcal{D}_f^\flat(x,y) =\infty$. 

It suffices to assume that $y\in \dom\partial f$ and $x\in \dom\partial f\subseteq \dom f$. Then, since $f(y) +f^\ast(v) =\langle y,v \rangle$ for all $v\in \partial f(y)$, we derive that
\begin{subequations}
\begin{align}
\mathcal{D}_{f \oplus f^\ast}^{\flat}(x,y) 
&=\inf_{v\in \partial f(y)} \left(f(x) +f^\ast(v) -\langle x,v \rangle \right)\\
&=\inf_{v\in \partial f(y)} \left(f(x) -f(y) +\langle y,v \rangle -\langle x,v\rangle \right)\\
&= f(x) -f(y) +\inf_{v\in \partial f(y)} \langle y-x,v \rangle =\mathcal{D}_f^\flat(x,y).
\end{align}
\end{subequations}
Similarly, $\mathcal{D}_{f \oplus f^\ast}^{\sharp}(x,y) =\mathcal{D}_f^\sharp(x,y)$. The conclusion follows.
\end{proof}

We recall now the following results regarding the lower semicontinuity of the left and right distances; these apply to each of our computed examples.

\begin{lemma}[{\cite[Lemma~3.17]{BM-L18}}]
\label{lem:leftlsc} 
Let $y \in \dom T$. Then the following hold:
\begin{enumerate}
\item 
The function $\mathcal{D}_T^{\flat, h}(\cdot,y): X \rightarrow \RR_\infty$ is lsc at every $x \in \inte \dom (S)$ with respect to the strong topology in $X$ provided that $Tz$ is weakly closed for any $z$ in its domain; 
\item\label{lem:leftlsc2} The function $\mathcal{D}_T^{\sharp,h}(\cdot,y) : X \rightarrow \RR_\infty$ is lsc at every $x \in \dom (S)$ with respect to the strong topology in $X$.
\end{enumerate}
\end{lemma}

\begin{lemma}[{\cite[Lemma~3.18]{BM-L18}}]
\label{lem:rightlsc} 
Supose that $T$ is locally bounded in the interior of its domain and that the graph of $T$ is closed with respect to the strong-weak topology. Fix $y \in \inte\dom T$ and $x \in \dom S$. Then the function $\mathcal{D}_T^{\flat, h} : X \rightarrow \RR_{\infty}$ is lsc at $y$ with respect to the strong topology in $X$.
\end{lemma}

\begin{remark}[The lower closed distance]\label{rem:lsc}
Notice that in Lemma~\ref{lem:leftlsc}\ref{lem:leftlsc2}, $\mathcal{D}_T^{\star,h}(\cdot,y)$ may not be lower semicontinuous at $x \in \overline{\dom} \partial f \setminus \dom \partial f$, a case we will encounter in our examples. Notice also that in Lemma~\ref{lem:rightlsc}, for $y \notin \inte\dom T$, the distance may \emph{not} be lower semicontinuous with respect to the second variable, a phenomenon we will encounter in our examples. 

For these two reasons, we also introduce the notion of the \emph{lower closed GBD}, denoted by $\overline{\mathcal{D}}_T^{\star, h}$, which satisfies
\begin{equation}
\epi \overline{\mathcal{D}}_T^{\star, h} = \overline{\epi} \mathcal{D}_T^{\star, h},
\end{equation}
where $\star$ may be either $\flat$ or $\sharp$. The lower closed GBD is the lower semicontinuous regularization of the function $\mathcal{D}_T^{\star, h}$, as described in \cite{RW}; its direct formula is given by
$$
\overline{\mathcal{D}}_T^{\star, h}(x,y) := \underset{(x',y') \rightarrow (x,y)}{\liminf} {\mathcal{D}}_T^{\star, h}(x',y').
$$
The lower closed distances $\overline{\mathcal{D}}_{F_S}$, $\overline{\mathcal{D}}_{\sigma_{S}}$, and $\overline{\mathcal{D}}_{f\oplus f^\ast}$ are defined analogously.
\end{remark}

\section{How to compute generalized Bregman distances}\label{sec:examples}

\begin{example}[Energy]\label{ex:dist:energy}
Let $f: x \mapsto \frac{1}{2}x^2$ be the \emph{energy}. If we have $h_{\Id}$ as the Fitzpatrick function for $\partial f=\Id$, then our GBD distance is 
\begin{equation}
\mathcal{D}_{F_{\Id}}(x,y)=\frac14(x-y)^2,
\end{equation}
which is equivalent to a scaled version of the usual Moreau distance. On the other hand, the largest element of $\mathcal{H}(\Id)$ is just
\begin{equation}
\sigma_{\Id} (x,y) = \begin{cases} x^2 & x=y,\\
\infty & \text{otherwise}.		
\end{cases}
\end{equation}
One can obtain this result by computing $F_{\Id}^*$ straight from the definition of the conjugate and using Fact~\ref{fact:handy}\ref{handy:sigmafitz}. One can also obtain this result by using Fact~\ref{fact:handy}\ref{handy:sigma}, because the graph of $\Id$ is simply the diagonal. The corresponding distance is
\begin{equation}
\mathcal{D}_{\sigma_{\Id}} (x,y) = \begin{cases} 0 & x=y,\\
\infty & \text{otherwise}.		
\end{cases}
\end{equation}
\end{example}

In \cite{BDL17}, the asymptotic properties of Bregman envelopes are illustrated using Bregman distances constructed from three functions. One of these was the \emph{energy} from Example~\ref{ex:dist:energy}, for which the Bregman proximity operator and envelope specialize to the Moreau case. While the choice of representative function $F_{\Id}$ is equivalent to the Moreau case up to a change in parameter, notice that the example $\mathcal{D}_{\sigma_{\Id}}$ illustrates that this is \emph{not} the case for \emph{any} choice of representative function. This is an important distinction in our context.

\subsection{Boltzmann--Shannon entropy}\label{subsec:BSentropy}

Another function whose translated version was early considered by Censor and Lent \cite{Censor1981}, and whose Bregman envelopes are studied in \cite{BDL17}, is the (negative) Boltzmann--Shannon entropy: 
\begin{equation}\label{def:entropy}
\ent:\RR \rightarrow \RR: \quad x\mapsto \begin{cases} x\log x -x & \text{if} \; x > 0, \\
0 & \text{if} \; x=0,\\
\infty & \text{otherwise.}\end{cases} \end{equation}

The Boltzmann-Shannon entropy is particularly important and natural to consider, because its derivative is $\log$, its conjugate is $\ent^* = \exp$, and its associated Bregman distance is the Kullback--Leibler divergence,
\begin{equation}
\mathcal{D}_{\ent}:(x,y) \mapsto \begin{cases}x(\log(x)-\log(y))-x+y & \text{if}\; y >0,\\
y & \text{if}\; y >0\;\text{and}\; x= 0,\\
\infty & \text{otherwise},\end{cases}
\end{equation}
which is frequently used as a measure of distance between positive vectors in information theory, statistics, and portfolio selection. The GBD associated with the Fenchel--Young representative $\ent \oplus \ent^* \in \mathcal{H}(\log)$ is
\begin{equation}
\mathcal{D}_{\ent \oplus \ent^*}:(x,y) \mapsto \begin{cases}x(\log(x)-\log(y))-x+y & \text{if}\; x,y >0,\\
\infty & \text{otherwise}.\end{cases}
\end{equation}
Thus it may be seen that the \emph{Bregman} distance of the \emph{Boltzmann--Shannon entropy} is the special case of the \emph{GBD} for the Fenchel--Young representative of the \emph{logarithm} function, except on the set $(\dom f \setminus \dom \partial f) \times \dom \partial f = \{0\}\times \left]0,\infty \right[$ (see Proposition~\ref{prop:fitz_bregman} and Remark~\ref{rem:lsc}). Its lower closure is given by\footnote{One may rewrite the first case as $y>0,x\geq0$ and omit writing the separate case $x=0,y>0$ in \eqref{def:Bregman_lower_closure}, as long as one remembers to use the convention that $0\log(0)=0$.} 
\begin{equation}\label{def:Bregman_lower_closure}
\overline{\mathcal{D}}_{\ent \oplus \ent^*}:(x,y) \mapsto \begin{cases}x(\log(x)-\log(y))-x+y & \text{if}\; y >0,x> 0,\\
y & \text{if}\; y>0,x=0,\\
0 &\text{if}\; y=x=0,\\
\infty & \text{otherwise}\end{cases}
\end{equation}
and is shown in Figure~\ref{D:bregman}, while the Fenchel--Young representative $\ent \oplus \ent^*$ is shown in Figure~\ref{rep:bregman}.

We will consider \emph{new} distances built from the maximally monotone operator $\log$, and compare these to the known special case of the Bregman distance for the Boltzmann-Shannon entropy. The corresponding Fitzpatrick function (as computed in \cite{BMcS06}) and shown in Figure~\ref{rep:fitz} is
\begin{equation}\label{fitz:ent}
F_{\log} :(x,y) \mapsto \begin{cases}
+\infty & \text{if}\; x<0,\\
\exp(y-1) & \text{if}\; x=0,\\
xy + x\left(\W(xe^{1-y})+\frac{1}{\W(xe^{1-y})} -2 \right) & \text{if}\; x>0,
\end{cases}
\end{equation}
where $\W$ is the real principal branch of the Lambert $\W$ function that satisfies $\W(x)e^{W(x)}=x$ on $\left[-1/e,\infty \right[$. See, for example, \cite{Knuth}. Its occurrences in convex analysis and its relationship to the Boltzmann--Shannon entropy have been discussed in, for example, \cite{BL2018,BL2017,Borwein2011}.

\begin{example}[GBD $\overline{\mathcal{D}}_{F_{\log}}$]
The corresponding (closed) GBD is 
\begin{align}\label{fitz:dist}
\overline{\mathcal{D}}_{F_{\log}}: \RR_+ \times \RR_+ & \rightarrow \RR_+ \cup \{\infty\}\\
(x,y) &\mapsto \begin{cases}
\infty & \text{if}\; x<0 \;\text{or}\; y<0\; \text{or}\; (x>0 \; \text{and}\; y=0),\\
ye^{-1} & \text{if}\; x=0\; \text{and}\; y\geq 0,\\
x\left(\W\left(\frac{xe}{y} \right)+\frac{1}{\W\left(\frac{xe}{y} \right)} -2 \right) & \text{otherwise},\nonumber.
\end{cases}
\end{align}
This distance is shown in Figure~\ref{D:fitz}.
\end{example}
\begin{proof}
Combining Definitions~\ref{D} and \ref{fitz:dist} with the fact that $\dom \log = \left]0,+\infty \right[$, we have
\begin{subequations}
\begin{align}
\mathcal{D}_{F_{\log}}(x,y) &= \begin{cases} F_{\log}(x,\log(y))-
\langle x,\log(y) \rangle & \text{if~} x,y \in \dom \log,\\
+\infty & \text{otherwise}\end{cases}\\
&= \begin{cases}+\infty & \text{if~} x\leq 0 \text{~or~} y\leq 0,\\
x\log(y)+ x\left(\W(xe^{1-\log(y)})+\frac{1}{\W(xe^{1-\log(y)})} -2 \right) -x\log(y) & \text{otherwise}. \end{cases}
\end{align}
\end{subequations} 
This simplifies, by a bit of arithmetic, to the form in \eqref{fitz:dist}, except on the set $\{0\}\times\left[0,\infty \right[$. Taking the closure of the epigraph admits $\overline{\mathcal{D}}_{F_{\log}}(0,y)=ye^{-1}$. This example is particularly interesting, because we see the loss of the \emph{left} lower semicontinuity property at $0$ because $0 \in \dom f \setminus \dom \log$, and we also see the loss of the \emph{right} lower semicontinuity property because $0 \notin \inte \dom \log$; see Remark~\ref{rem:lsc}.
\end{proof}

\subsection{Computation of a difficult representative function and the conjugate of a Fitzpatrick function}\label{sec:Fitzconj}

Next we consider the case where $S=\log$ is the logarithm function on $\left]0,\infty \right[$. Even though we know the form of $F_{\log}$, it is not straightforward to compute $\sigma_{\log}$ using the equality $\sigma_{\log}(x,y)={F}_{\log}^*(y,x)$ from Fact~\ref{fact:handy}\ref{handy:sigmafitz} by subdifferentiating with the latter and solving. Instead, we use the characterization from Fact~\ref{fact:handy}\ref{handy:sigma}.

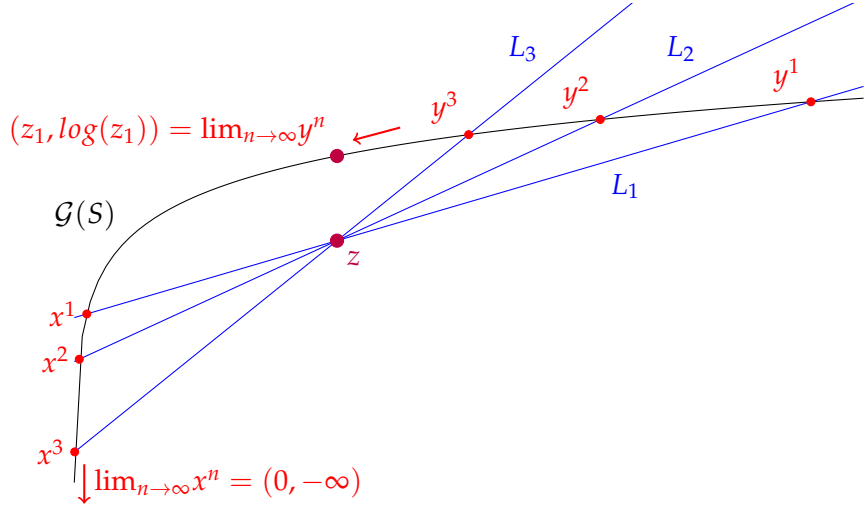
\begin{figure}
\begin{center}
\begin{tikzpicture}[scale=0.7]
\clip (-2,4.5) rectangle + (17,-10);

\draw plot [domain=0.01:15,samples=100] (\x,{ln(\x)});

\draw [blue] plot [domain=0.01:15,samples=100] (\x,{.2932*\x+-1.466)});
\draw [fill,red] (.248,-1.393) circle [radius=.075cm];
\node [left, red] at (.248,-1.393) {$x^1$};
\draw [fill,red] (14,2.639) circle [radius=.075cm];
\node [above left, red] at (14,2.639) {$y^1$};
\node [below right, blue] at (10,1.466) {$L_1$};

\draw [blue] plot [domain=0.01:15,samples=100] (\x,{.4605*\x+-2.302)});
\draw [fill,red] (0.105,-2.254) circle [radius=.075cm];
\node [left, red] at (0.105,-2.254) {$x^2$};
\draw [fill,red] (10,2.303) circle [radius=.075cm];
\node [above left, red] at (10,2.303) {$y^2$};
\node [above left, blue] at (12,3.224) {$L_2$}; 

\draw [blue] plot [domain=0.01:15,samples=100] (\x,{.806*\x+-4.030)});
\node [above left,blue] at (9,3.224) {$L_3$};
\draw [fill,red] (7.5,2.015) circle [radius=.075cm];
\node [above left, red] at (7.5,2.015) {$y^3$};
\draw [fill,red] (0.018,-4.015) circle [radius=.075cm];
\node [left, red] at (0.018,-4.015) {$x^3$};

\draw [fill,purple] (5,0) circle [radius=.125cm];
\node [below right, purple] at (5,0) {$z$};
\draw [fill,purple] (5,1.609) circle [radius=.125cm];
\node [above left, red] at (5,1.609) {$(z_1,log(z_1))={\rm lim}_{n\rightarrow \infty}y^n$};
\draw [<-,red,thick] (5.3,1.91) -- (6.2,2.15);

\draw [->,red,thick] (0.2,-4.25) -- (0.2,-5);
\node [right,red] at (0.2,-4.6) {${\rm lim}_{n \rightarrow \infty} x^n = (0,-\infty)$};

\node [above left] at (1,0) {$\mathcal{G}(S)$};
\end{tikzpicture}
\caption{Construction of sequence in proof of Theorem~\ref{thm:sigma} }\label{fig:sigmaproof}
\end{center}
\end{figure}

Recall that, for an arbitrary function $g$ and its \emph{convex hull} function ${\rm co}(g)$, the lower semicontinuous regularization or {\em lower closure}, denoted as $\overline{\rm co}(g)$, has the property
\begin{equation}\label{co}
{\rm epi}(\overline{\rm co}(g))=\overline{\rm co}({\rm epi}(g)).
\end{equation}
See, for example, \cite[Chapter~1]{RW}.

\begin{theorem}[The representative $\sigma_{\log}$]\label{thm:sigma} 
Let $T =S =\log$. Then
\begin{equation}\label{def:sigma_log}
\sigma_{\log}:(x,y) \mapsto \begin{cases}
x\log(x) &\text{if~} y \leq \log(x),\\
\infty & \text{otherwise},
\end{cases}
\end{equation}
whose graph is shown in Figure~\ref{rep:sigma}.
\end{theorem}
\begin{proof}
Using Fact~\ref{fact:handy} together with the fact that $\mathcal{G}(S) = \left \{(z_1,\log(z_1)) | z_1 \in \left]0,\infty \right[ \right \}$ and the fact that $\log$ is a concave function, we have that $z_2 > \log (z_1)$ implies $\sigma_{\partial f}(z_1,z_2) = \infty$. Indeed, let $g:\RR^2\to \RR_{\infty}$ be defined as $g:=\pi + \iota_{\mathcal{G}(S)}$; the graph of $g$ is shown as the dark curve at the boundary of the surface in Figure~\ref{rep:sigma}. By Fact~\ref{fact:handy} we have that 
$\sigma_{\partial f}=\overline{\rm co}(g)$. If $\overline{\rm co}(g)(z_1,z_2)<\infty$ then there exists $a\in \RR$ such that $(z_1,z_2,a)\in\epi \overline{\rm co}(g)$. By \eqref{co} this is equivalent to $(z_1,z_2,a)\in\overline{\rm co}\epi (g)$. The last inclusion means that there exists a sequence $w_n:=(z_1^n,z_2^n,a_n)\in {\rm conv}\epi (g)$ such that $(z_1,z_2,a)=\lim_{n\to \infty} (z_1^n,z_2^n,a_n)$. Note that we can assume that
\[
(z_1^n,z_2^n,a_n)=\sum_{i=1}^4 \lambda_{n,i} (z_{1,i}^n,z_{2,i}^n,a_{i,n}), \hbox{ with } (z_{1,i}^n,z_{2,i}^n,a_{i,n})\in \epi (g),
\]
thanks to Carath\'{e}odory's theorem. Using the fact that $(z_{1,i}^n,z_{2,i}^n,a_{i,n})\in \epi (g)$, we have that
\[
\begin{array}{ll}
\sum_{i=1}^4 \lambda_{n,i}=1,\,\lambda_{n,i}\ge 0,& \,\forall \, i=1,\ldots,4,\\
&\\
z_1^n=\sum_{i=1}^4 \lambda_{n,i} z_{1,i}^n\,,\,& z_2^n=\sum_{i=1}^4 \lambda_{n,i} z_{2,i}^n\,,\\
&\\
z_{1,i}^n>0, &\, z_{2,i}^n=\log{z_{1,i}^n}\,, \forall \, i=1,\ldots,4,\\
\end{array}
\]
Using the above expression for $z_2^n$ and the fact that $z_{2,i}^n=\log{z_{1,i}^n}\,$, we can write
\[
z_2^n=\sum_{i=1}^4 \lambda_{n,i} z_{2,i}^n=\sum_{i=1}^4 \lambda_{n,i} \log{z_{1,i}^n}\le \log{\left( \sum_{i=1}^4 \lambda_{n,i} {z_{1,i}^n}\right)}=\log{z_1^n},
\]
where we used the fact that $\log({\cdot})$ is concave. Taking limits and using the continuity of the $\log(\cdot)$ we deduce that $z_2 \leq \log(z_1)$. This implies that, when $z_2 > \log(z_1)$ we must have $\sigma_{\partial f}(z_1,z_2) = \infty$. This shows the second part of the definition in the statement of the theorem. We proceed now to prove the first part of the definition of $\sigma_{\partial f}$. Let $z \in \RR^2$ be such that $z_1>0$ and $z_2 \leq \log(z_1)$. For any $x,y \in \mathcal{G}(S)$ that satisfy $\lambda x + (1-\lambda)y = z$ for some $\lambda \in [0,1]$, we have that $g(x)=\langle x_1, x_2 \rangle = x_1x_2$ and that $g(y)=\langle y_1 ,y_2 \rangle = y_1y_2$. Thus 
\begin{subequations}
\begin{align}
\left(x_1,x_2, \langle x_1 x_2 \rangle \right) &= (x_1,x_2,x_1x_2) \in \epi g,\\
\text{and}\quad \left(y_1,y_2, \langle y_1 y_2 \rangle \right) &= (y_1,y_2,y_1y_2) \in \epi g.
\end{align}
\end{subequations} 
From the definition of convexity,
\begin{equation}
\lambda(x_1,x_2,x_1 x_2) + (1-\lambda)(y_1,y_2,y_1y_2) \in \conv \left(\epi g\right).
\end{equation}
This is just

\begin{equation}
(z_1,z_2,\lambda x_1 x_2 + (1-\lambda) y_1 y_2) \in \conv \left(\epi g\right).
\end{equation}
Using the fact that $x,y \in \mathcal{G}(S)$, this is just
\begin{equation}
\left(z_1,z_2,\lambda x_1 \log(x_1) + (1-\lambda) y_1 \log(y_1)\right) \in \conv \left(\epi g\right).
\end{equation}
Let $(\varphi_n)_{n \in \NN} \subset \left]0,\pi/2 \right[$ be a sequence that satisfies $\lim_{n \rightarrow \infty}\varphi_n = \pi/2$. For any $\varphi_n$, we may find a line in $\RR^2$ that goes through $z$ and has slope $\tan(\varphi_n)$, which is given by
\begin{equation}
L_n := \{u \in \RR^2 \;|\; u_2 = \tan(\varphi_n)(u_1 - z_1) + z_2  \}.
\end{equation}
Now $\varphi_n \in \left]0,\pi/2 \right[$ and $z_2 \leq \log(z_1)$ guarantees that $L_n \cap \mathcal{G}(S)$ is a doubleton $\{x^n,y^n\}$ where $x_1^n < z_1$ and $y_1^n > z_1$ and $z=p_n x^n + q_n y^n$ with $q_n + p_n = 1, p_n,q_n \in [0,1]$. The construction of this sequence is shown in Figure~\ref{fig:sigmaproof}. As $\varphi_n \rightarrow \pi/2$, the slope $\tan(\varphi_n)$ of $L_n$ goes to infinity, and so we have that
\begin{subequations}
\begin{align}
\underset{n \rightarrow \infty}{\lim}& y^n = (z_1,\log(z_1))\quad \text{and}\quad \underset{n \rightarrow \infty}{\lim} x^n = (0,-\infty),\\
\text{and so}\quad \underset{n \rightarrow \infty}{\lim}& x^n \log(x^n) =0, \quad \text{and}\quad \underset{n \rightarrow \infty}{\lim} p_n = 0, \quad \text{and} \quad \underset{n \rightarrow \infty}{\lim}q_n = 1.
\end{align}
\end{subequations}
Thus we have that 
\begin{equation}
\underset{n \rightarrow \infty}{\lim}\left(z_1,z_2, p_nx^n\log(x^n)+q_ny^n\log(y^n) \right) = \left(z_1,z_2, z_1 \log(z_1)\right) \in \overline{\rm co}(\epi g)=\epi \sigma_{\log}.
\end{equation}
Thus $z_1\log(z_1)\ge \sigma_{\partial f}(z)$ for every $(z_1,z_2)$ such that $z_1>0$ and $z_2 \leq \log(z_1)$. For the converse inequality, define the function
\begin{equation}
w(t_1,t_2):=\left\{ \begin{array}{lr}
t_1\, \log{t_1}  & \hbox{ if } t_1>0,\\
0& \hbox{ if } t_1=0.
\end{array}\right.
\end{equation}
The function $w$ is convex and lsc.  It is easy to check that $w\le g$ in $\RR^2$. Therefore,
\begin{equation}
\epi w\supset \epi g.
\end{equation}
Using the fact that $w$ is convex and lsc we deduce that
\begin{equation}
\epi w=\overline{\rm co}(\epi w) \supset \overline{\rm co}(\epi g)=\epi \sigma_{\log},
\end{equation}
equivalently, $w\le \sigma_{\log}$. This implies that $z_1\log(z_1)\le \sigma_{\log}(z)$ for every $(z_1,z_2)$ such that $z_1>0$ and $z_2 \leq \log(z_1)$. Consequently, we showed that 
\begin{equation}
\sigma_{\log}(z)=z_1\log(z_1)\,,\forall (z_1,z_2) \hbox{ s.t. } z_1>0 \hbox{ and } z_2 \leq \log(z_1),
\end{equation}
which is the claim of the theorem.
\end{proof}

\begin{figure}
\begin{center}
\subfloat[{$F_{\log}$}\label{rep:fitz}]{\includegraphics[width=.3\textwidth]{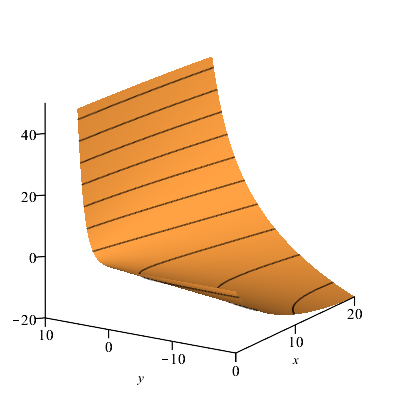}}
\subfloat[{$\ent \oplus \ent^*$}\label{rep:bregman}]{\includegraphics[width=.3\textwidth]{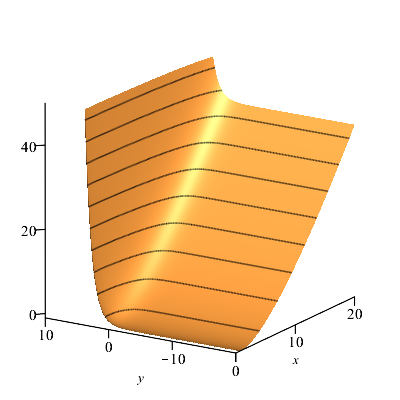}}
\subfloat[{$\sigma_{\log}$}\label{rep:sigma}]{\includegraphics[width=.3\textwidth]{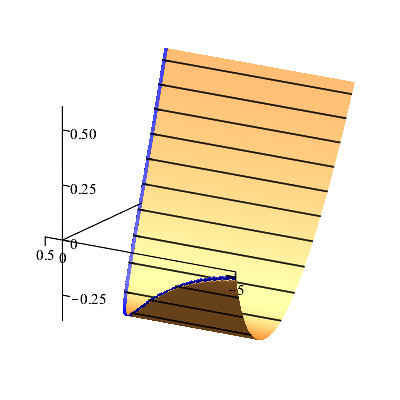}}
\end{center}
\caption{Constructing representative functions for the logarithm.}\label{fig:representativefunctions}
\end{figure}

\begin{corollary}[The conjugate of the Fitzpatrick function {$F_{\log}$} ] We have that
\begin{equation}
F_{\log}^*:(x,y) \mapsto \begin{cases}
y\log(y) &\text{if~} x \leq \log(y),\\
\infty & \text{otherwise}.
\end{cases}
\end{equation}
\end{corollary}
\begin{proof}
This immediately follows from Theorem~\ref{thm:sigma} together with Fact~\ref{fact:handy}\ref{handy:sigmafitz}.
\end{proof}

\begin{remark}
In the proof of Theorem~\ref{thm:sigma}, the sequences $x^n,y^n$ may be given explicitly by
\begin{subequations}
\begin{align}
x_1^n &= \exp \left(-\W_0 \left(-\tan(\varphi_n)\exp\left(-z_1 \tan(\varphi_n)+z_2 \right) \right) \right),\\
y_1^n &= \exp \left(-\W_{-1} \left(-\tan(\varphi_n)\exp\left(-z_1 \tan(\varphi_n)+z_2 \right) \right) \right),
\end{align}
\end{subequations}
where $\W_0$ and $\W_{-1}$ are the principal and secondary real branches of the Lambert $\W$ function.
\end{remark}
Most of the analysis of Lambert $\W$ in the context of convex optimization has focused on its principal branch. However, in order to experimentally discover the true form for $\sigma_{\partial f}$ from Theorem~\ref{thm:sigma}, we had to make use of both real branches. The reason for this is that our attempts to explicitly solve the systems
\begin{subequations}
\begin{align}
\underset{y \in \RR^2}{\sup} &\left \{ \langle x,y \rangle - F_{\partial f}(y) \right \}\\
\text{or} \quad \underset{\begin{subarray}{c}
z=\lambda x + (1-\lambda)y \\
x,y \in G(\partial f),\;\; \lambda \in [0,1]
\end{subarray}}{\inf} &\left \{\lambda x_1 x_2 + (1-\lambda)y_1y_2  \right \}
\end{align}
\end{subequations}
were not successful. Seeking to compute $\sigma_{\log}$ numerically, we constructed a numerical procedure that evaluated $\lambda x_1^n x_2^n + (1-\lambda)y_1^n y_2^n$ for a finite sequence $(\varphi_n)_{n=1}^N$ and chose the smallest value to represent $\sigma_{\log}(z)$. The fast evaluation of Lambert $\W$ obviated the implementation of slower numerical routines to solve the equation system
\begin{equation}
\log(\eta) = \tan(\varphi_n) (\eta-z_1)+z_2.
\end{equation}
We observed that the smallest value was always the last value, corresponding to $\varphi_n$ nearest to $\pi/2$. Once we observed that the values were consistently approaching $z_1 \log(z_1)$ for any $z$ chosen, we ``knew'' the true form. Upon further scrutiny of the geometry, we realized that the sequences that led to the discovery also yielded the proof.

\begin{corollary}[GBD for $\sigma_{\log}$]\label{cor:sigmalog}The corresponding (closed) conjugate Fitzpatrick distance is
\begin{align}\label{eqn:Dsigma}
\overline{\mathcal{D}}_{\sigma_{\log}}: (x,y) \mapsto  \begin{cases}
x \log(x) - x\log(y) & \text{if~}\ 0< y \leq x,\\
0 & \text{if~} x=y=0,\\
\infty & \text{otherwise},
\end{cases}
\end{align}
which is shown in Figure~\ref{D:sigma}.
\end{corollary}
\begin{proof}
From the definition and Theorem~\ref{thm:sigma}:
\begin{subequations}
\begin{align}
\mathcal{D}_{\sigma_{\log}}: (x,y) \mapsto \;\;&\sigma_{\log}(x,\log(y))-\langle x, \log (y) \rangle\\
&= - x\log(y)+\begin{cases}
x \log(x) & \text{if~} \log(y)\leq \log(x),\\
\infty & \text{otherwise},
\end{cases}
\end{align}
\end{subequations}
which may be recognized as the form in \eqref{eqn:Dsigma} except at the point $(0,0)$. Taking the closure of the epigraph of $\mathcal{D}_{\sigma_{\log}}$, we obtain $\overline{\mathcal{D}}_{\sigma_{\log}}(0,0)=0$. This example is illustrative, because lower semicontinuity is lost, but only at the point $(0,0)$, since $0 \notin \dom \log$; see Remark~\ref{rem:lsc}.
\end{proof}

\begin{figure}
\begin{center}
\subfloat[{$\overline{\mathcal{D}}_{F_{\log}}$}\label{D:fitz}]{\includegraphics[width=.3\textwidth]{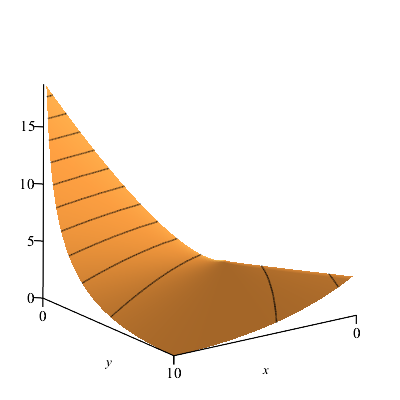}}
\subfloat[{$\overline{\mathcal{D}}_{\ent \oplus \ent^*}$}\label{D:bregman}]{\includegraphics[width=.3\textwidth]{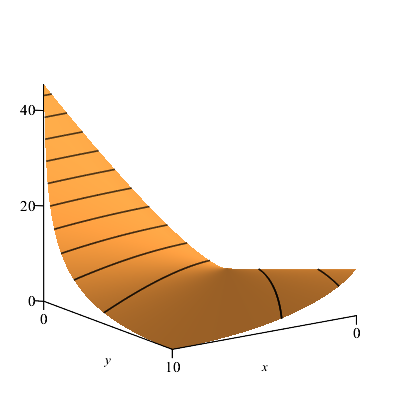}}
\subfloat[{$\overline{\mathcal{D}}_{\sigma_{\log}}$}\label{D:sigma}]{\includegraphics[width=.3\textwidth]{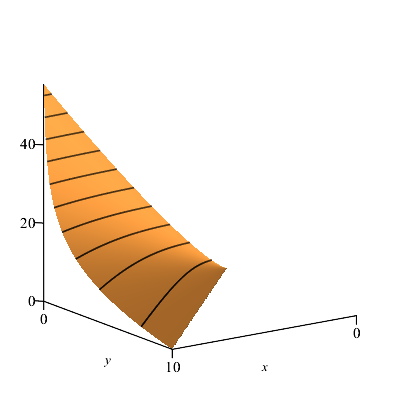}}
\end{center}
\caption{Distances constructed from $\mathcal{H}(\log)$.}\label{fig:D}
\end{figure}

\section{A coercivity framework for generalized Bregman distances}\label{sec:properties}

In this section, we will establish important properties of the GBD. From now on, $X$ is a \emph{reflexive} real Banach space. For the sake of simplicity, when we make use of the norms $\| \cdot \|_{X\oplus X^*}$, $\| \cdot \|_{X}$, and $\| \cdot \|_{X^*}$, we allow context to make clear which norm is being used. 

\subsection{Convexity}

\begin{proposition}
\label{p:cvx}
Let $x\in \dom S$ and $y\in \dom T$. Then the following hold:
\begin{enumerate}
\item\label{p:cvx_x}$\mathcal{D}_T^{\sharp,h}(\cdot,y)$ is convex.
\item\label{p:cvx_y} 
If $T$ is an affine mapping, then $\mathcal{D}_T^h(x,\cdot)$ is convex.
\end{enumerate}
\end{proposition}
\begin{proof}

We first note that $h$ is convex on $X\times X^\ast$ by definition.

\ref{p:cvx_x}: As $h(\cdot, v) -\langle \cdot,v \rangle$ is convex for all $v\in Ty$, it follows from \cite[Proposition~3.4.3(iii)]{BI08} that $\mathcal{D}_T^{\sharp,h}(\cdot,y)$ is convex. 

\ref{p:cvx_y}: Since $T$ is single-valued, we have for all $y\in \dom T$ that
\begin{equation}
\mathcal{D}_T^h(x,y) =\mathcal{D}_T^{\flat,h}(x,y) =\mathcal{D}_T^{\sharp,h}(x,y) =h(x,Ty) -\langle x,Ty \rangle.
\end{equation}
Therefore, $\mathcal{D}_T^h(x,\cdot) =(h(x,\cdot)-\langle x,\cdot \rangle)\circ T$ is convex because it is the composition of a convex function with an affine function; see, e.g., \cite[Lemma~2.1.8(b)]{BV10}.
\end{proof}

\subsection{Coercivity and supercoercivity}\label{sec:coercivity}

We now turn our attention to coercivity and supercoercivity. These properties of distances are important, because they are essential to the analysis of associated envelopes and proximity operators. After first providing a framework for verifying these properties of the GBDs, we will show in Section~\ref{sec:coercivityofthesum} how these properties admit corresponding coercivity properties for the sum of the GBDs together with Legendre functions. These results on sums are the key to analysing the envelopes; see \cite[Lemma 2.12]{BCN06} and \cite{BDL17}.

From now on, as mentioned, we assume our spaces to be reflexive, so that we may make use of the following fact from \cite[Remark 3.12]{BM-L18}.

\begin{fact}[{\cite[Remark 3.12]{BM-L18}}]
\label{f:estimate}
When $X$ is a reflexive space, it holds that
\begin{equation}
\forall x,y\in X,\quad \mathcal{D}_T^{\flat,h}(x,y) \geq \frac{1}{4}\underset{v \in Ty}{\inf}d^2 \left((x,v),\mathcal{G}(S) \right) = \frac14 d^2 \left(\{x\} \times Ty,\mathcal{G}(S) \right),
\end{equation}
where $d$ denotes the distance on $X\times X^*$ defined by $d((x,v),(y,w)) :=\sqrt{\|x-y\|^2+\|v-w\|^2}$. Consequently, we can see $\mathcal{D}_T^{\flat, h}(x,y)$ as providing us with an upper estimate of the distance between the sets $\{x\} \times Ty$ and $\mathcal{G}(S)$.
\end{fact}

Throughout this section, we exploit the fact that the GBD is minorized by the distance between the sets $\{x\} \times Ty$ and $\mathcal{G}(S)$ in order to establish left and right coercivity and supercoercivity of the distance. The intuition behind the results is shown in Figure~\ref{fig:sqrtx}.

The following elementary lemma will be useful for our analysis.
\begin{lemma}
\label{l:lim}
Let $(x_n)_{n\in \NN}$, $(y_n)_{n\in \NN}$, and $(z_n)_{n\in \NN}$ be sequences in $X$ such that $\|x_n\|\to \infty$ as $n\to \infty$. Then the following hold:
\begin{enumerate}
\item\label{l:lim_1}
Suppose that $\|z_{k_n}\|\to \infty$ whenever $(y_{k_n})_{n\in \NN}$ is a subsequence of $(y_n)_{n\in \NN}$ with $\|y_{k_n}\|\to \infty$. Then, for all $\alpha\in \RR_{++}$, 
\begin{equation}
\|x_n-y_n\|^\alpha +\|z_n\|^\alpha\to \infty \quad\text{as~} n\to\infty.
\end{equation}
\item\label{l:lim_2}
Suppose that $\|z_{k_n}\|^2/\|y_{k_n}\|\to \infty$ whenever $(y_{k_n})_{n\in \NN}$ is a subsequence of $(y_n)_{n\in \NN}$ with $\|y_{k_n}\|\to \infty$. Then 
\begin{equation}
\frac{\|x_n-y_n\|^2 +\|z_n\|^2}{\|x_n\|}\to \infty \quad\text{as~} n\to\infty.
\end{equation} 
\end{enumerate}
\end{lemma}
\begin{proof}
\ref{l:lim_1}: Suppose to the contrary that there exist subsequences $(x_{k_n})_{n\in \NN}$, $(y_{k_n})_{n\in \NN}$, and $(z_{k_n})_{n\in \NN}$ such that the sequence $(\|x_{k_n}-y_{k_n}\|^\alpha +\|z_{k_n}\|^\alpha)_{n\in \NN}$ is bounded. Then both $(x_{k_n}-y_{k_n})_{n\in \NN}$ and $(z_{k_n})_{n\in \NN}$ are bounded.  
By assumption, passing to another subsequence if necessary, we obtain that the sequence $(y_{k_n})_{n\in \NN}$ is also bounded, and so is $(x_{k_n})_{n\in \NN}$ since
\begin{equation}
\forall n\in \NN,\quad \|x_{k_n}\|\leq \|x_{k_n}-y_{k_n}\| +\|y_{k_n}\|.
\end{equation}
This contradicts the assumption that $\|x_n\|\to \infty$. 

\ref{l:lim_2}: Suppose that there exist subsequences $(x_{k_n})_{n\in \NN}$, $(y_{k_n})_{n\in \NN}$, $(z_{k_n})_{n\in \NN}$ and a constant $\mu >0$ such that 
\begin{equation}
\label{e:contrary}
\forall n\in \NN,\quad \frac{\|x_{k_n}-y_{k_n}\|^2 +\|z_{k_n}\|^2}{\|x_{k_n}\|} <\mu.
\end{equation}
Then, by Cauchy--Schwarz inequality, 
\begin{subequations}
\label{e:ykn}
\begin{align}
2\|y_{k_n}\|&\geq \frac{2\langle x_{k_n},y_{k_n} \rangle}{\|x_{k_n}\|} =\frac{\|x_{k_n}\|^2 +\|y_{k_n}\|^2 -\|x_{k_n}-y_{k_n}\|^2}{\|x_{k_n}\|}\\
&\geq \|x_{k_n}\| -\frac{\|x_{k_n}-y_{k_n}\|^2}{\|x_{k_n}\|} >\|x_{k_n}\| -\mu. 
\end{align}
\end{subequations}
As $n\to \infty$, since $\|x_n\|\to \infty$, it follows from \eqref{e:ykn} that $\|y_{k_n}\|\to \infty$ and, by assumption, $\|z_{k_n}\|^2/\|y_{k_n}\|\to \infty$.
On the other hand, combining \eqref{e:contrary} with \eqref{e:ykn} yields
\begin{equation}
\frac{\|z_{k_n}\|^2}{\|y_{k_n}\|} <\mu\frac{\|x_{k_n}\|}{\|y_{k_n}\|}\leq \mu\frac{2\|y_{k_n}\|+\mu}{\|y_{k_n}\|} =2\mu +\frac{\mu^2}{\|y_{k_n}\|}\to 2\mu.
\end{equation}  
A contradiction is thus obtained, and we complete the proof.
\end{proof}

\begin{remark}
\label{r:precoer}
\begin{enumerate}
\item 
Since $\mathcal{D}_T^{\sharp,h}(x,y)\geq \mathcal{D}_T^{\flat,h}(x,y)$ for all $(x,y)\in \dom S\times \dom T$, if $\mathcal{D}_T^{\flat,h}$ is coercive or supercoercive with respect to the first or second variable, then so is $\mathcal{D}_T^{\sharp,h}$. We will thus focus on the coercivity and supercoercivity of $\mathcal{D}_T^{\flat,h}$. 
\item\label{r:precoer_exact} 
Assume that $Ty$ is compact with respect to the strong topology. Then $\{x\}\times Ty$ is also compact. This, together with Fact~\ref{f:estimate} and the fact that $\mathcal{G}(S)$ is closed, allows us to choose $v\in Ty$ and $(a,b)\in \mathcal{G}(S)$ such that 
\begin{equation}
\|x-a\|^2 +\|v-b\|^2 =d^2\big((x,v),(a,b)\big) =d^2(\{x\}\times Ty, \mathcal{G}(S))\leq 4\mathcal{D}_T^{\flat,h}(x,y).
\end{equation} 
\end{enumerate}
\end{remark}

\begin{theorem}[Left coercivity and left supercoercivity of $\mathcal{D}_T^{\flat,h}$]
\label{t:left}
Let $y\in \dom T$. Then
\begin{enumerate}
\item\label{t:left_dom} 
If $\dom S$ is bounded, then $\mathcal{D}_T^{\flat,h}(\cdot,y)$ is supercoercive and hence coercive.
\end{enumerate}
Suppose further that $Ty$ is compact with respect to the strong topology. Then the following hold:
\begin{enumerate}
\stepcounter{enumi}
\item\label{t:left_coer} 
If $S$ is \emph{coercive} in the sense that $(a_n,b_n)\in \mathcal{G}(S)$ and $\|a_n\|\to \infty$ imply $\|b_n\|\to \infty$, then $\mathcal{D}_T^{\flat,h}(\cdot,y)$ is coercive.
\item\label{t:left_super}
If $(a_n,b_n)\in \mathcal{G}(S)$ and $\|a_n\|\to \infty$ imply $\|b_n\|^2/\|a_n\|\to \infty$, then $\mathcal{D}_T^{\flat,h}(\cdot,y)$ is supercoercive.
\end{enumerate}
\end{theorem}
\begin{proof}
Let $(x_n)_{n\in \NN}$ satisfy $\|x_n\|\to \infty$ as $n\to \infty$.  

\ref{t:left_dom}: As $\dom S$ is bounded, there exists $N\in \NN$ such that for $n\geq N$ we have $x_n\notin \dom S$. Fixing an arbitrary $n\geq N$, by definition, $\mathcal{D}_T^{\flat,h}(x_n,y) =\infty$, so $\mathcal{D}_T^{\flat,h}(x_n,y)/\|x_n\| =\infty$, and we are done.	

To prove \ref{t:left_coer} and \ref{t:left_super}, we derive from Remark~\ref{r:precoer}\ref{r:precoer_exact} that, since $Ty$ is compact, there exist $v_n\in Ty$ and $(a_n,b_n)\in \mathcal{G}(S)$ such that
\begin{equation}
\label{e:estimateD}
\forall n \in \NN,\quad 4\mathcal{D}_T^{\flat,h}(x_n,y)\geq \|x_n-a_n\|^2 +\|v_n-b_n\|^2.
\end{equation}
Here, we note that $\|x_n\|\to \infty$ as $n\to \infty$ and that $(v_n)_{n\in \NN}$ is bounded due to compactness of $Ty$. 

\ref{t:left_coer}: If $(a_{k_n})_{n\in \NN}$ is a subsequence of $(a_n)_{n\in \NN}$ with $\|a_{k_n}\|\to \infty$, then by assumption \ref{t:left_coer}, $\|b_{k_n}\|\to \infty$, which implies that $\|v_{k_n}-b_{k_n}\|\to \infty$. Applying Lemma~\ref{l:lim}\ref{l:lim_1} to the sequences $(x_n)_{n\in \NN}$, $(a_n)_{n\in \NN}$, and $(v_n-b_n)_{n\in \NN}$, we obtain that $\|x_n-a_n\|^2 +\|v_n-b_n\|^2\to \infty$ and, by \eqref{e:estimateD}, $\mathcal{D}_T^{\flat,h}(x_n,y)\to \infty$ as $n\to \infty$.  

\ref{t:left_super}: If $(a_{k_n})_{n\in \NN}$ is a subsequence of $(a_n)_{n\in \NN}$ with $\|a_{k_n}\|\to \infty$, then by assumption \ref{t:left_super}, $\|b_{k_n}\|^2/\|a_{k_n}\|\to \infty$, so $\|b_{k_n}\|\to \infty$ and
\begin{equation}
\frac{\|v_{k_n}-b_{k_n}\|^2}{\|a_{k_n}\|}\geq \frac{\|b_{k_n}\|^2-2\|v_{k_n}\|\|b_{k_n}\|}{\|a_{k_n}\|} =\frac{\|b_{k_n}\|^2}{\|a_{k_n}\|} \left(1-\frac{2\|v_{k_n}\|}{\|a_{k_n}\|\|b_{k_n}\|}\right)\to \infty.
\end{equation} 
Now, using Lemma~\ref{l:lim}\ref{l:lim_2} yields 
\begin{equation}
\frac{\|x_n-a_n\|^2 +\|v_n-b_n\|^2}{\|x_n\|}\to \infty \quad\text{as~} n\to \infty, 
\end{equation}  
which together with \eqref{e:estimateD} completes the proof.
\end{proof}

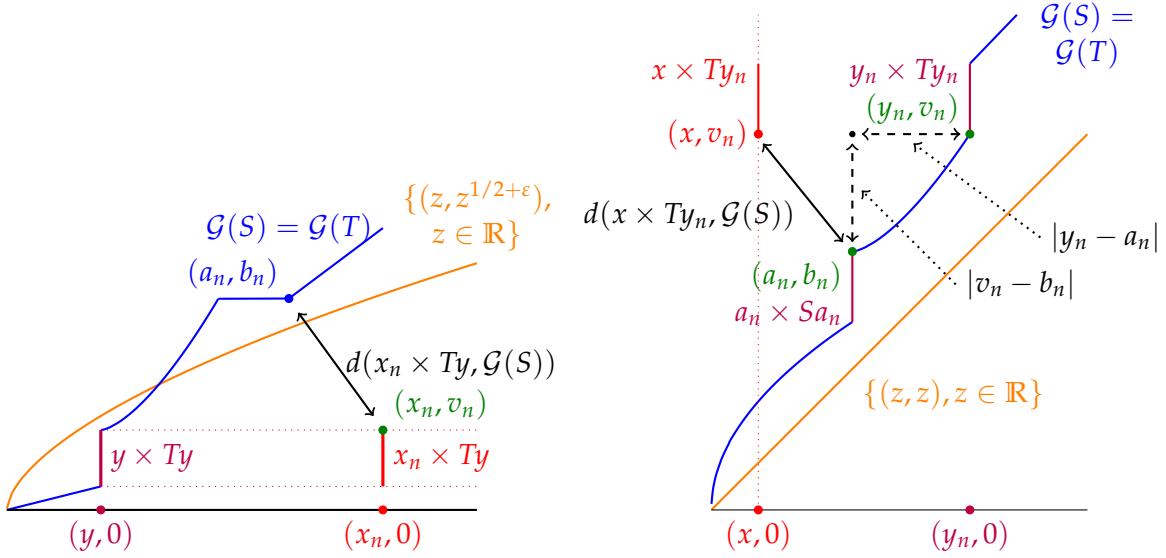
\begin{figure}
\begin{center}
\begin{tikzpicture}[scale=1.25]
\draw[thick,scale=1,domain=0.001:5,smooth,samples=100,variable=\x,orange] plot ({\x},{\x^(6/10)});
\node [above,orange] at (5,2.62) {$\begin{array}{c} \{(z,z^{1/2+\varepsilon}),\\
	z \in \RR \}\end{array}$};

\draw [thick,blue] (0,0) -- (1,.25);
\draw[thick,scale=1,domain=0.001:1.25,smooth,samples=50,variable=\x,blue] plot ({\x+1},{\x^(3/2)+.85});
\draw [thick,blue] (2.25,2.24754) -- (3,2.25);
\node [left, blue] at (4,3) {$\mathcal{G}(S)=\mathcal{G}(T)$};
\draw [thick,blue] (3,2.25) -- (4,3);

\draw [thick,black] (0,0) -- (5,0);

\draw [purple,dotted] (1,.25) -- (5,.25);
\draw [purple,dotted] (1,.85) -- (5,.85);

\draw [very thick,purple] (1,.25) -- (1,.85);
\node [right,purple] at (1,.55) {$y \times Ty$};
\filldraw [purple] (1,0) circle(1.2pt);
\node [below,purple] at (1,0) {$(y,0)$};

\draw [very thick,red] (4,.25) -- (4,.85);
\node [right,red] at (4,.55) {$x_n \times Ty$};
\filldraw [red] (4,0) circle(1.2pt);
\node [below,red] at (4,0) {$(x_n,0)$};	

\filldraw [color={rgb:black,1;green,1}] (4,.85) circle(1.2pt);
\node [above right,color={rgb:black,1;green,1}] at (4,.85) {$(x_n,v_n)$};

\draw [<->,thick,black] (3.1,2.10) -- (3.9,1);
\node [right,black] at (3.5,1.55) {$d(x_n \times Ty , \mathcal{G}(S)) $};

\filldraw [blue] (3,2.25) circle(1.2pt);
\node [above left,blue]  at (3,2.25) {$(a_n,b_n) $};

\end{tikzpicture}\begin{tikzpicture}[scale=1.25]

\draw[thick,scale=1,domain=0.001:4,smooth,samples=100,variable=\x,orange] plot ({\x},{\x});
\node [below right,orange] at (1.5,1.5) {$\{(z,z),z \in \RR \}$};

\draw[thick,scale=1,domain=0.001:1,smooth,samples=100,variable=\x,blue] plot ({(1.5)*\x},{(2.0)*\x^(1/2)});
\draw[thick,scale=1,domain=0.001:1,smooth,samples=100,variable=\x,blue] plot ({(1.25)*\x+1.5},{(1.25)*\x^(3/2)+2.75});
\draw[thick,blue] (2.75,4.75) -- (3.25,5.27);
\node [right,blue] at (3.25,5.05) {$\begin{array}{c}\mathcal{G}(S)=\\ \mathcal{G}(T)\end{array}$};

\draw[black] (0,0) -- (4,0);
\filldraw [red] (0.5,0) circle(1.2pt);
\node [red,below] at (0.5,0) {$(x,0)$};
\filldraw [purple] (2.75,0) circle(1.2pt);
\node [purple,below] at (2.75,0) {$(y_n,0)$};
\draw [red,dotted] (0.5,0) -- (0.5,5.27);

\draw[thick,purple] (1.5,2.0) -- (1.5,2.75);
\node [left,purple] at (1.5,2.1) {$a_n \times S a_n$};
\filldraw [color={rgb:black,1;green,1}] (1.5,2.75) circle(1.2pt);
\node [below left,color={rgb:black,1;green,1}] at (1.5,2.75) {$(a_n,b_n)$};

\draw[thick,purple] (2.75,4) -- (2.75,4.75); 
\node [left,purple] at (2.75,4.65) {$y_n \times Ty_n$};
\filldraw [color={rgb:black,1;green,1}] (2.75,4) circle(1.2pt);
\node [above left,color={rgb:black,1;green,1}] at (2.75,4) {$(y_n,v_n)$};

\draw[thick,red] (.5,4) -- (.5,4.75); 
\node [left,red] at (.5,4.65) {$x\times Ty_n$};
\filldraw [red] (0.5,4) circle(1.2pt);
\node [left,red] at (0.5,4) {$(x,v_n)$};	

\filldraw [black] (1.5,4) circle(0.8pt);

\draw [thick,dashed, <->,black] (2.65,4) -- (1.6,4);
\draw [thick,dotted,->,black] (3.5,2.9) -- (2.15,3.9);
\node [right,black] at (3.5,2.9) {$|y_n-a_n|$};

\draw [thick,dashed, <->,black] (1.5,2.85) -- (1.5,3.9);
\draw [thick,dotted,->,black] (2.6,2.4) -- (1.6,3.4);
\node [right,black] at (2.6,2.4) {$|v_n-b_n|$};

\draw [black,thick,<->] (0.55,3.9) -- (1.4,2.85);
\node [below left,black] at (1,3.4) {$d(x\times Ty_n,\mathcal{G}(S))$};
\end{tikzpicture}
\end{center}
\caption{Left: the motivation for Theorem~\ref{t:left} is exemplified by Example~\ref{ex:sqrtx}. Right: the motivation for Theorem~\ref{t:right} is exemplified by Example~\ref{ex:xsquared}.}\label{fig:sqrtx}
\end{figure}

\begin{theorem}[Right coercivity and right supercoercivity of $\mathcal{D}_T^{\flat,h}$]
\label{t:right}
Let $x\in \dom S$. Then 
\begin{enumerate}
\item\label{t:right_dom} 
If $\dom T$ is bounded, then $\mathcal{D}_T^{\flat,h}(x,\cdot)$ is supercoercive and hence coercive.
\end{enumerate}
Suppose further that $T$ has strongly compact images. Then the following hold:
\begin{enumerate}
\stepcounter{enumi}
\item\label{t:right_coer} 
If $(a_n,b_n)\in \mathcal{G}(S)$, $(y_n,v_n)\in \mathcal{G}(T)$, and $\|y_n-a_n\|\to \infty$ imply $\|v_n-b_n\|\to \infty$, then $\mathcal{D}_T^{\flat,h}(x,\cdot)$ is coercive.
\item\label{t:right_super} 
If $(a_n,b_n)\in \mathcal{G}(S)$, $(y_n,v_n)\in \mathcal{G}(T)$, and $\|y_n-a_n\|\to \infty$ imply $\|v_n-b_n\|^2/\|y_n-a_n\|\to \infty$, then $\mathcal{D}_T^{\flat,h}(x,\cdot)$ is supercoercive.
\end{enumerate}
\end{theorem}
\begin{proof}
Let $(y_n)_{n\in \NN}$ satisfy $\|y_n\|\to \infty$ as $n\to \infty$. 

\ref{t:right_dom}: By the boundedness of $\dom T$, there exists $N\in \NN$ such that for $n\geq N$, $y_n\notin \dom T$. Fixing $n\geq N$, the definition of $\mathcal{D}_T^{\flat,h}$ yields $\mathcal{D}_T^{\flat,h}(x,y_n) =\infty$, which implies that $\mathcal{D}_T^{\flat,h}(x,y_n)/\|y_n\| =\infty$, and we are done.

\ref{t:right_coer} \& \ref{t:right_super}: Since $T$ has strongly compact images, $Ty_n$ is compact with respect to the strong topology. By Remark~\ref{r:precoer}\ref{r:precoer_exact}, there exist $v_n\in Ty_n$ and $(a_n,b_n)\in \mathcal{G}(S)$ such that
\begin{equation}
\forall n \in \NN,\quad 4\mathcal{D}_T^{\flat,h}(x,y_n)\geq \|x-a_n\|^2 +\|v_n-b_n\|^2 =\|(y_n-x)-(y_n-a_n)\|^2 +\|v_n-b_n\|^2.
\end{equation}
As $n\to \infty$, $\|y_n-x\|\to \infty$ since $\|y_n\|\to \infty$. 

Now we show \ref{t:right_coer}. Suppose that $(a_n,b_n)\in \mathcal{G}(S)$, $(y_n,v_n)\in \mathcal{G}(T)$, and $\|y_n-a_n\|\to \infty$ imply $\|v_n-b_n\|\to \infty$. We have that if $(y_{k_n}-a_{k_n})_{n\in \NN}$ is a subsequence of $(y_n-a_n)_{n\in \NN}$ with $\|y_{k_n}-a_{k_n}\|\to \infty$, then $\|v_{k_n}-b_{k_n}\|\to \infty$. Applying Lemma~\ref{l:lim}\ref{l:lim_1} to the sequences $(y_n-x)_{n\in \NN}$, $(y_n-a_n)_{n\in \NN}$, and $(v_n-b_n)_{n\in \NN}$, we obtain that $\|(y_n-x)-(y_n-a_n)\|^2 +\|v_n-b_n\|^2\to \infty$, and so $\mathcal{D}_T^{\flat,h}(x,y_n)\to \infty$ as $n\to \infty$. This shows \ref{t:right_coer}.

Now we show \ref{t:right_super}. Suppose that $(a_n,b_n)\in \mathcal{G}(S)$, $(y_n,v_n)\in \mathcal{G}(T)$, and $\|y_n-a_n\|\to \infty$ imply $\|v_n-b_n\|^2/\|y_n-a_n\|\to \infty$. We derive that if $(y_{k_n}-a_{k_n})_{n\in \NN}$ is a subsequence of $(y_n-a_n)_{n\in \NN}$ with $\|y_{k_n}-a_{k_n}\|\to \infty$, then $\|v_{k_n}-b_{k_n}\|^2/\|y_{k_n}-a_{k_n}\|\to \infty$. Now, Lemma~\ref{l:lim}\ref{l:lim_2} completes the proof. This shows \ref{t:right_super}.
\end{proof}

As we will see in the following example, the conditions in Theorem~\ref{t:left} (resp.\ Theorem~\ref{t:right}) are not necessary conditions for the left (resp.\ right) coercivity or supercoercivity of $\mathcal{D}_T^{\flat,h}$.
\begin{example}
Suppose that $X =\RR$. Let $f =\Id\colon \RR\to \RR$, $S =\nabla f =1$, and $h\colon \RR\times \RR\to \RR_{\infty}$ given by
\begin{equation}
h(x,v) =f(x) +f^\ast(v) =x +\iota_{\{1\}}(v).
\end{equation}
Let also $T =0$. Then 
\begin{equation}
\forall (x,y)\in \RR^2,\quad \mathcal{D}_T^h(x,y) =h(x,0) -\langle x,0 \rangle =h(x,0) =\infty.
\end{equation}
Therefore, both $\mathcal{D}_T^h(\cdot,y)$ and $\mathcal{D}_T^h(x,\cdot)$ are supercoercive and hence coercive (for all $x,y\in \RR$), while $S$ and $T$ do not satisfy the assumptions in Theorem~\ref{t:left} nor in Theorem~\ref{t:right}.
\end{example}

\begin{corollary}[Left supercoercivity of $\mathcal{D}_h$]
\label{cor:leftsupercoercivesub}
Let $f \in \Gamma_0(X)$ be such that $\partial f$ is point-to-point, $T =S =\partial f$, $h\in \mathcal{H}(S)$, and $y\in \dom S$. Then the following hold:
\begin{enumerate}
\item If $\dom S$ is bounded,  $\mathcal{D}_{h}(\cdot,y)$ is supercoercive and hence coercive.
\item If $S$ satisfies the property that $(a_n,b_n) \in \mathcal{G}(S)$ and $\|a_n\| \rightarrow \infty$ implies $\| b_n\| \rightarrow \infty$, then $\mathcal{D}_{h}(\cdot,y)$ is coercive.
\item If $(a_n,b_n) \in \mathcal{G}(S)$ and $\|a_n\| \rightarrow \infty$ imply $\|b_n\|^2/\|a_n\| \rightarrow \infty$, then $\mathcal{D}_{h}(\cdot,y)$ is supercoercive.
\end{enumerate}
\end{corollary}
\begin{proof}
Apply Theorem~\ref{t:left} with $T=S=\partial f$. Because $f \in \Gamma_0(\mathcal{H})$, we have that $\partial f$ is maximally monotone. The compactness of $\partial f(y)$ comes from the fact that $\partial f$ is point-to-point. 
\end{proof}

\begin{example}[Left supercoercive distances on $\RR$]\label{ex:sqrtx}
Let $X = \RR$ and $f:=x \mapsto |x|^{3/2+\varepsilon}$ for some $\varepsilon >0$, and let $h \in \mathcal{H}(\nabla f)$. Then $\mathcal{D}_{h}$ is left supercoercive.

To check, we need only show that $f$ satisfies the criteria for  Corollary~\ref{cor:leftsupercoercivesub}.

Since $\nabla f: x \mapsto (3/2+\varepsilon){\rm sign}(x)|x|^{1/2+\varepsilon}$, we have that
\begin{equation}
\frac{\|\nabla f(x)\|^2}{\|x\|} \geq \frac{|x|^{1+2\varepsilon}}{|x|} = |x|^{2\varepsilon} \rightarrow \infty \;\; \text{as} \;\; |x| \to \infty,
\end{equation}
showing the sufficient conditions for Corollary~\ref{cor:leftsupercoercivesub}. This example is illustrated in Figure~\ref{fig:sqrtx}, which shows the geometric intuition underpinning Theorem~\ref{t:left}.
\end{example}

\begin{corollary}[Right supercoercivity of $\mathcal{D}_{h}$]
\label{cor:rightsupercoercivesub} 
Let $f \in \Gamma_0(X)$ be such that $\partial f$ is point-to-point, $T =S =\partial f$, $h\in \mathcal{H}(\partial f)$, and $x\in \dom \partial f$. Then the following hold:
\begin{enumerate}
\item
If $\partial f$ is bounded, then $\mathcal{D}_{h}$ is supercoercive and hence coercive.
\item
If $(a_n,b_n),(y_n,v_n)\in \mathcal{G}({\partial f})$ and $\|y_n-a_n\|\to \infty$ imply $\|v_n-b_n\|\to \infty$, then $\mathcal{D}_{h}$ is coercive.
\item
If $(a_n,b_n),(y_n,v_n)\in \mathcal{G}(\partial f)$ and $\|y_n-a_n\|\to \infty$ imply $\|v_n-b_n\|^2/\|y_n-a_n\|\to \infty$, then $\mathcal{D}_{h}$ is supercoercive.
\end{enumerate}
\end{corollary}
\begin{proof}
Apply Theorem~\ref{t:right} with $T=S=\partial f$. Here $\partial f$ automatically has compact images because it is point-to-point. 
\end{proof}

\begin{example}[Right supercoercive distances on $\RR$]\label{ex:xsquared}
Let $X = \RR$ and $f:=x \mapsto |x|^{p}$ for some $p \geq 2$ and $h \in \mathcal{H}(\partial f)$. Then $\mathcal{D}_{h}$ is right supercoercive.

To check, we need only show that $f$ satisfies the criteria for  Corollary~\ref{cor:rightsupercoercivesub}. 

Since $\nabla f: x \rightarrow {\rm sign}(x)p|x|^{p-1}$ and $p\geq 2$ we have that $(|x-y| \geq 2) \implies \|\nabla f(x)-\nabla f(y)\| \geq |x-y|$. Thus, for $|x-y| \geq 2$,
\begin{equation}
\frac{\| \nabla f (x) - \nabla f(y)\|^2}{\|y-x\|} \geq \frac{|y-x|^2}{|y-x|} = |y-x| \rightarrow \infty \;\;\text{as}\;\; |y-x| \rightarrow \infty,
\end{equation}
showing the sufficient conditions for Corollary~\ref{cor:rightsupercoercivesub}. This example is illustrated at right in Figure~\ref{fig:sqrtx}.
\end{example}

\begin{example}[Functions on $\RR$ which fail the assumptions of Theorems~\ref{t:left} and \ref{t:right}]
Let $X = \RR$ and $f:=x \mapsto |x|^{3/2}$. 

Let $x_n:=n$ and $y_n:=0$. Then, since $\nabla f: x \mapsto {\rm sign}(x)|x|^{1/2}$, we have that
\begin{equation}
d((x_n,\nabla f(y_n)),\mathcal{G}(\nabla f)) = d((n, 0),\mathcal{G}(\nabla f)) \leq d((n, 0),(n,n^{1/2}))=n^{1/2}=x_n^{1/2},
\end{equation}
and so $d((x_n,\nabla f(y_n),\mathcal{G}(\nabla f))^2 = x_n$ for all $n$.
\end{example}

\begin{example}[Theorem conditions sufficient but not necessary]
The conditions of Theorem~\ref{t:left} are sufficient but not necessary. Let $f$ be the Boltzmann-Shannon entropy, and we have that for $x>1$:
\begin{equation}
\frac{\|\nabla f(x)\|^2}{\|x\|}=\frac{\log(x)^2}{x} \leq \frac{x}{x} = 1 \rightarrow 1 \;\;\text{as} \;\;x \rightarrow \infty,
\end{equation}
so the sufficient conditions from Theorem~\ref{t:left} fail. 

The form of $\mathcal{D}_{\sigma_{\log}}$ is given in \eqref{eqn:Dsigma}. If $x$ or $y$ is less than or equal zero, $\mathcal{D}_{\sigma_{\log}}(x,y)=\infty$. Fixing $y>0$, we have that for $x>1$:
\begin{equation}
\frac{\mathcal{D}_{\sigma_{\partial f}}(x,y)}{\|x\|} = \frac{x(\log(x)-\log(y))}{x} = \log(x)-\log(y) \rightarrow \infty \;\;\text{as}\;\; x \rightarrow \infty,
\end{equation}
and so $\mathcal{D}_{\sigma_{\log}}$ is left supercoercive. 
\end{example}

\subsection{Coercivity of the sum of $\mathcal{D}$ and a convex function}\label{sec:coercivityofthesum}

The following propositions and their accompanying proofs extend and follow the template of Bauschke, Combettes, and Noll in \cite[Lemma 2.12]{BCN06}, with modifications necessary in order to handle the greater generality of $\mathcal{D}_T^{\flat,h}$. In the following, $X$ is assumed to be a real Hilbert space, $U_S := \inte \dom S$, and $U_T := \inte \dom T$. 

\begin{proposition}[Left coercivity of the sum of $\mathcal{D}_T^{\flat,h}$ and a convex function]
\label{p:lsum}
Let $\theta\in \Gamma_0(X)$ be such that $U_S\cap \dom\theta\neq \varnothing$ and let $\gamma\in \RR_{++}$. Suppose that one of the following holds:
\begin{enumerate}[label=(\alph*)]
\item\label{p:lsum_dom} 
$U_S\cap \dom\theta$ is bounded and for all $y\in U_T$, $\mathcal{D}_T^{\flat,h}(\cdot,y)$ is coercive.
\item\label{p:lsum_inf} 
$\inf \theta(U_S) >-\infty$ and for all $y\in U_T$, $\mathcal{D}_T^{\flat,h}(\cdot,y)$ is coercive.
\item\label{p:lsum_super} 
For all $y\in U_T$, $\mathcal{D}_T^{\flat,h}(\cdot,y)$ is supercoercive.
\end{enumerate}
Then
\begin{equation}
\label{e:lsum}
\forall y\in U_T,\quad \theta(\cdot) +\frac{1}{\gamma}\mathcal{D}_T^{\flat,h}(\cdot,y) \text{~is coercive}.
\end{equation}
\end{proposition}
\begin{proof}
We will show that \ref{p:lsum_dom} $\implies$ \ref{p:lsum_inf} $\implies$ \eqref{e:lsum} and that \ref{p:lsum_super} $\implies$ \eqref{e:lsum}. 
First, we have from \cite[Theorem 9.20]{BC17} that there exists $(u,\alpha)\in X\times \RR$ such that 
\begin{equation}
\theta\geq \langle u,\cdot \rangle +\alpha.
\end{equation}

\ref{p:lsum_dom} $\implies$ \ref{p:lsum_inf}: By Cauchy--Schwarz inequality,
\begin{equation}
\forall x\in X,\quad \theta(x)\geq \langle u,x \rangle +\alpha\geq -\|u\|\|x\| +\alpha,
\end{equation}
which yields
\begin{equation}
\inf\theta(U_S) =\inf\theta(U_S\cap \dom\theta)\geq -\|u\| \sup_{x\in U_S\cap \dom\theta} \|x\| +\alpha > -\infty
\end{equation}
since $U_S\cap \dom\theta$ is bounded. Hence, \ref{p:lsum_dom} $\implies$ \ref{p:lsum_inf}.

\ref{p:lsum_inf} $\implies$ \eqref{e:lsum}: Let $y\in U_T$. Suppose for a contradiction that there exist a sequence $(x_n)_{n \in N}$ in $X$ and a constant $\mu\in \RR_{++}$ such that $\|x_n\|\to \infty$ and  
\begin{equation}
\label{e:lsum_bounded}
\forall n\in \NN,\quad \theta(x_n) +\frac{1}{\gamma}\mathcal{D}_T^{\flat,h}(x_n,y) \leq \mu.
\end{equation}
For each $n\in \NN$, since $\theta(x_n) >-\infty$, it follows that $\mathcal{D}_T^{\flat,h}(x_n,y) <\infty$, and so $x_n\in \dom S$. Next, according to \cite[Proposition~11.1(iv)]{BC17}, $\inf\theta(\dom S) =\inf\theta(U_S)$, which implies that $\theta(x_n)\geq \inf\theta(U_S) >-\infty$ for all $n\in \NN$. Combining with \eqref{e:lsum_bounded}, we obtain that
\begin{equation}
\forall n\in \NN,\quad \mathcal{D}_T^{\flat,h}(x_n,y) \leq \gamma(\mu -\inf\theta(U_S)) <\infty,
\end{equation}
which contradicts the coercivity of $\mathcal{D}_T^{\flat,h}(\cdot,y)$.

\ref{p:lsum_super} $\implies$ \eqref{e:lsum}: Notice that
\begin{equation}
\theta(\cdot) +\frac{1}{\gamma}\mathcal{D}_T^{\flat,h}(\cdot,y)\geq \langle u,\cdot \rangle +\alpha +\frac{1}{\gamma}\mathcal{D}_T^{\flat,h}(\cdot,y).
\end{equation}
The right-hand side is the sum of a supercoercive function and an affine function, and hence a coercive function due to \cite[Corollary 16.21]{BC17}. Since $\theta(\cdot) +\frac{1}{\gamma}\mathcal{D}_T^{\flat,h}(\cdot,y)$ is bounded from below by a coercive function, it is coercive. 
\end{proof}

\begin{proposition}[Right coercivity of the sum of $\mathcal{D}_T^{\flat,h}$ and a convex function]
Let $\theta\in \Gamma_0(X)$ be such that $U_T\cap \dom\theta\neq \varnothing$ and let $\gamma\in \RR_{++}$. Suppose that one of the following holds:
\begin{enumerate}[label=\textbf{(\alph*)}]
\item
$U_T\cap \dom\theta$ is bounded and for all $x\in U_S$, $\mathcal{D}_T^{\flat,h}(x,\cdot)$ is coercive.
\item
$\inf\theta(U_T) >-\infty$ and for all $x\in U_S$, $\mathcal{D}_T^{\flat,h}(x,\cdot)$ is coercive.
\item
For all $x\in U_S$, $\mathcal{D}_T^{\flat,h}(x,\cdot)$ is supercoercive.
\end{enumerate}
Then
\begin{equation}\label{e:rsum}
\forall x\in U_S,\quad \theta(\cdot) +\frac{1}{\gamma}\mathcal{D}_T^{\flat,h}(x,\cdot) \text{~is coercive}.
\end{equation}
\end{proposition}
\begin{proof}
This is analogous to the proof of Proposition~\ref{p:lsum}.
\end{proof}

\section{Conclusion}\label{sec:conclusion}

In Section~\ref{sec:preliminaries}, we illuminated the similarities between Bregman distances and the new GBDs, explaining the domain conditions under which they are equal when the Fenchel--Young representative is employed. We also introduced the lower closed GBD, a variant whose advantages we motivated in Sections~\ref{sec:examples} and \ref{sec:properties}.

In Section~\ref{sec:examples}, we provided detailed examples of how to compute the new GBDs, illustrating with the energy and the Boltzmann--Shannon entropy, whose Bregman distances respectively correspond to the classical Moreau case and the Kullback--Leibler divergence. We compared the Fenchel--Young representative case with the two cases of the Fitzpatrick representative and its conjugate. These are the two other most natural representative functions to consider, because they serve as book-ends for the representative set $\mathcal{H}(S)$, as motivated in Section~\ref{sec:preliminaries}. 

In Section~\ref{sec:Fitzconj} we answered the open question of finding the conjugate for the Fitzpatrick function of the logarithm. In so-doing, we demonstrated how to use the graphical characterizations of representative functions in order to compute GBDs, and we illustrated the role that special functions like Lambert $\W$ play in computational discovery. The method of computational discovery that we used is prototypical of what one might employ in similar situations where the symbolic computation poses a challenge.

Section~\ref{sec:properties} contains the most important theoretical contribution of this work: a framework for verifying the coercivity and supercoercivity of the left and right distances, as well as the coercivity of the sum of these distances together with a Legendre function. We have also illustrated how this framework for sufficiency possesses a useful geometric interpretation, because the GBDs provide an upper estimate on a set distance. In our examples, we illustrated what might go wrong when sufficient criteria do not hold. These coercivity properties are important, because of the role they play in establishing asymptotic properties for envelopes and proximity operators in the classical Bregman case, and also in establishing existence of minimizers of regularized problems; see, for example, \cite{BDL17,BurDut10,BI98,BI99,BS01}. Such properties are important, because many optimization algorithms may be viewed as special cases of gradient descent applied to envelope functions.

\subsubsection*{Future work}

The coercivity framework we have established makes possible several new avenues of inquiry. While the conditions we provide for verifying coercivity and supercoercivity in Section~\ref{sec:properties} are sufficient, they are not always necessary. An important future work is to catalogue useful (computable) distances for which the coercivity results hold. In particular, by establishing the aforementioned coercivity framework, we have set the table for a study of the left and right envelopes, along with their corresponding proximity operators. A much more interesting question is whether certain optimization algorithms might be viewed as gradient descent applied to GBD envelopes other than already-known Fenchel--Young cases. Another natural question is: what do the dual characterizations of such algorithms look like?

\subsection*{Acknowledgements}

The authors are grateful to Yair Censor for his warm comments on an early version of this manuscript, and for his helpful and detailed historical remarks on Bregman distances. We also thank the two anonymous referees for their careful comments and suggestions, which resulted in an improvement of the original presentation. MND was partially supported by the Australian Research Council (ARC) Discovery Project DP160101537. MND visited University of South Australia in 2018; this visit was instrumental to this work, and he acknowledges their hospitality. SBL was supported by an Australian Mathematical Society Lift-Off Fellowship and Hong Kong Research Grants Council PolyU153085/16p.

\end{document}